\documentclass[10pt]{article}
\usepackage{amsfonts}
\usepackage{dsfont}
\usepackage{amsmath,mathrsfs}
\usepackage{amsthm}
\usepackage{amssymb}
\usepackage{amscd}
\usepackage{graphicx}
\usepackage{float}
\usepackage{subfigure}
\usepackage{epstopdf}
\usepackage{appendix}
\usepackage{booktabs}
\usepackage{units}
\usepackage{enumerate}
\usepackage{mathrsfs}
\usepackage{longtable}
\allowdisplaybreaks[1]
\newtheorem{theorem}{Theorem}[section]

\newtheorem{lemma}{Lemma}[section]
\newtheorem{proposition}{Proposition}[section]
\newtheorem{remark}{Remark}[section]

\newtheorem{problem}{Problem}[section]
\newtheorem*{problem(P1)}{Problem (P1)}
\newtheorem*{problem(P2)}{Problem (P2)}
\newtheorem*{problem(P1a)}{Problem (P1a)}
\newtheorem*{problem(P1b)}{Problem (P1b)}
\newtheorem*{problem(P2a)}{Problem (P2a)}

\begin{document}
\title{{Linear-quadratic stochastic nonzero-sum differential games between graphon teams}\thanks{De-xuan Xu, Zhun Gou, and Nan-jing Huang: The work of these authors was supported by the National Natural Science Foundation of China (12171339).}}
\author{{De-xuan Xu$^a$, Zhun Gou$^b$ and Nan-jing Huang$^a$\footnote{Corresponding author,  E-mail: nanjinghuang@hotmail.com; njhuang@scu.edu.cn} }\\
{\scriptsize\it a. Department of Mathematics, Sichuan University, Chengdu,
Sichuan 610064, P.R. China}\\
{\scriptsize\it b. College of Mathematics and Statistics, Chongqing Technology and Business University, Chongqing 400067, P.R. China}}
\date{}
\maketitle
\vspace*{-9mm}
\begin{center}
\begin{minipage}{5.8in}
{\bf Abstract.} We study a class of nonzero-sum stochastic differential games between two teams with agents in each team interacting through graphon aggregates. On the one hand, in each large population group, agents act together to optimize a common social cost function. On the other hand, these two groups compete with each other, forming a Nash game between two graphon teams. We note that the original problem can be equivalently formulated as an infinite-dimensional two-agent Nash game. Applying the dynamic programming approach, we obtain a set of coupled operator-valued Riccati-type equations. By proving the existence of solutions to the equations mentioned above, we obtain a Nash equilibrium for the two teams.
\\ \ \\
{\bf 2020 Mathematics Subject Classification:}
\\ \ \\
\end{minipage}
\end{center}
\section{Introduction}
\qquad Let $[0,T]$ be a finite-time horizon for fixed $T>0$, $(\Omega, \mathcal{F}, \mathbb{P})$ be a complete probability space. Denote $\mathcal{H}\triangleq L^2[0,1]$. The norm and inner product in $\mathcal{H}$ are denoted by $\|\cdot\|$ and $\langle\cdot,\cdot\rangle$.

In this work, we study a class of nonzero-sum games between two continuous populations of interacting agents subject to stochastic forcing, which are described by the following dynamics:
\begin{equation}\label{4-1}
  \begin{split}
    dx^1(t,\alpha)= & \left\{A_1x^1(t,\alpha)+B_1u^1(t,\alpha)+D_1\int_0^1M_1(\alpha,\beta)x^1(t,\beta)d\beta+\varepsilon F_1\int_0^1x^2(t,\beta)d\beta\right\}dt+\sigma_1dW^1(t,\alpha), \\
   dx^2(t,\alpha)= & \left\{A_2x^2(t,\alpha)+B_2u^2(t,\alpha)+D_2\int_0^1M_2(\alpha,\beta)x^2(t,\beta)d\beta+\varepsilon F_2\int_0^1x^1(t,\beta)d\beta\right\}dt+\sigma_2dW^2(t,\alpha),
  \end{split}
\end{equation}
where $x^1(t,\cdot), u^1(t,\cdot), x^2(t,\cdot), u^2(t,\cdot)$ are $\mathcal{H}$-valued stochastic processes on $[0,T]\times\Omega$, denoting the states and inputs for each team, respectively, $A_i$, $B_i$, $D_i$, $F_i$, $\sigma_i$, $\varepsilon$, $i=1,2$ are given constants, $M_1(\cdot,\cdot), M_2(\cdot,\cdot)$ are two symmetric measurable function from $[0,1]^2$ to $[-1,1]$, named ``graphons'' \cite{Gao1,Gao4,Dexuan}, describing the strength of heterogenous interactions among agents within each group, the initial conditions $x^1(0,\cdot), x^2(0,\cdot)$ are assumed to be deterministic for simplicity, $W^1, W^2$ are two independent $\mathcal{H}$-valued $Q$-Wiener processes with covariance operators $\mathbf{Q}_1, \mathbf{Q}_2$ and the following expansions:
\begin{equation*}
  W_t^1=\sum_{k=1}^\infty \sqrt{\lambda_k^1}B_t^{1,k}e_k^1,\quad W_t^2=\sum_{k=1}^\infty \sqrt{\lambda_k^2}B_t^{2,k}e_k^2,
\end{equation*}
and where $\{e_k^1\}$, $\{e_k^2\}$ are orthonormal sets in $L^2[0,1]$ diagonalizing $\mathbf{Q}_1, \mathbf{Q}_2$, respectively, $\{\lambda_k^1\}$, $\{\lambda_k^2\}$ are the corresponding eigenvalues, $\{B^{1,k}\}$, $\{B^{2,k}\}$ are mutually independent real valued Brownian motions. Equation \eqref{4-1} is a model that describes two continuous populations of interacting agents \cite{Medvedev1,Medvedev2,Kaliuzhnyi,Medvedev}.

Denote $\mathcal{F}_t^W=\sigma\{W_s^1,W_s^2\ s\leq t\}\bigvee\mathcal{N}$. The admissible strategy set of the $i$th team (i=1, 2) is given by
\begin{equation*}
  \mathcal{U}_{two}=\left\{u^i|u^i\ \mbox{is}\ \mathcal{F}^W \mbox{-progressively measurable and satisfies}\ E\int_0^T \|u_t^i\|^2dt<\infty\right\}.
\end{equation*}

The cost functions of the teams are defined by
\begin{align*}
  J^1(u^1,u^2)= & \frac{1}{2}E\int_0^T\int_0^1\left\{\bar{Q}_1\left|x^1(t,\alpha)-\Gamma_1\int_0^1M_1(\alpha,\beta)x^1(t,\beta)d\beta\right|^2+R_{11}\big|u^1(t,\alpha)\big|^2\right.\\
  & \left.+\varepsilon R_{12}\big|u^2(t,\alpha)\big|^2\right\}d\alpha dt+\frac{1}{2}E\int_0^1\bar{Q}_{1f}\big|x^1(T,\alpha)\big|^2d\alpha
\end{align*}
and
\begin{align*}
  J^2(u^1,u^2)= & \frac{1}{2}E\int_0^T\int_0^1\left\{\bar{Q}_2\left|x^2(t,\alpha)-\Gamma_2\int_0^1M_2(\alpha,\beta)x^2(t,\beta)d\beta\right|^2+R_{22}\big|u^2(t,\alpha)\big|^2\right.\\
  & \left.+\varepsilon R_{21}\big|u^1(t,\alpha)\big|^2\right\}d\alpha dt+\frac{1}{2}E\int_0^1\bar{Q}_{2f}\big|x^2(T,\alpha)\big|^2d\alpha,
\end{align*}
where $\bar{Q}_i\geq 0$, $\bar{Q}_{if}\geq 0$, and $R_{ii}>0$ for $i=1,2$.

We now consider the Nash game problem for two continuous populations.
\begin{problem}\label{twonash}
Seek $u^i\in\mathcal{U}_{two}$, $i=1,2$, such that
$$
  J^1(u^1,u^2)\leq J^1(v^1,u^2),\quad J^2(u^1,u^2)\leq J^2(u^1,v^2)
$$
for any admissible control $v^1,v^2$ in $\mathcal{U}_{two}$. The strategies $u=(u^1,u^2)$ is called a two-team Nash equilibrium.
\end{problem}

It is worth noticing that Problem \ref{twonash} is one kind of stochastic games problem among teams. To date, several researchers have investigated problems of this type. For example, Sanjari et al. \cite{Sanjari} studied Nash equilibrium problems between two teams of finite and infinite agents, respectively, in both static and discrete-time dynamic settings, where the interactions among agents within each team are symmetric. Zaman et al. \cite{Zaman} investigated reinforcement learning in nonzero-sum games involving multiple teams. Their work also focused on discrete-time models, with homogeneous agents within each team. Carmona et al. \cite{Carmona7} discussed a class of mean-field type zero-sum games involving two decision-makers.

To the best knowledge of the authors, there is no research investigating the Nash game problems between two teams with hetergeneous agents in continuous time. Motivated by this, we analyze a class of Nash games between two teams with agents interacting through graphon aggregates. In contrast to \cite{Sanjari,Zaman}, our model is set in continuous time and allows for asymmetric interactions among agents within each team. It is noted that, in our model \eqref{4-1}, we use nonlocal diffusion equations to describe the dynamics of two populations, where each contains a continuum of agents. There is another common framework consisting of a continuum of coupled systems, which also models interacting diffusions (see, e.g., \cite{Coppini}).

We obtain in the sequel that Problem \ref{twonash} is essentially an infinite-dimensional two-agent Nash game. As far as we are aware, the study of infinite-dimensional stochastic differential games remains at an early stage, with relatively few results available in the current literature. Fouque and Zhang \cite{Fouque} reformulated a stochastic game with delay as an infinite-dimensional stochastic game and solved both the corresponding infinite-dimensional $N$-player game as well as the mean field game. More recently, \cite{Liu,Ghilli,Federico} investigated infinite-dimensional mean field games.
We remark that although this work and \cite{Fouque} both utilize the dynamic programming principle to solve the infinite-dimensional game problems, the approach adopted in this paper is different. For the infinite-dimensional two-agent Nash game considered in this work, we employ the dynamic programming principle to reformulate the infinite-dimensional Nash game as a problem of existence of solutions to a system of coupled operator-valued Riccati-type equations. By establishing the existence of solutions to this set of equations, we obtain a solution to the original Nash game problem between two teams.

The organization of this paper is as follows. In Section 2, we rewrite the two-team Nash game as an infinite-dimensional two-agent Nash game, and obtain a Nash equilibrium solution which depends on a set of coupled operator-valued Riccati-type equations. In Section 3, we prove the existence of solutions to the Riccati-type equations, before we summarize the results in Section 4.

\hspace*{\fill}\\
\noindent

\section{Two-team Nash equilibrium}
\qquad Now we propose the definition of a graphon operator $T_M:L^2[0,1]\to L^2[0,1]$ as follows \cite{Lovasz}
\begin{equation}\label{graphonoperator}
(T_M\varphi)(\alpha)=\int_{[0,1]}M(\alpha,\beta)\varphi(\beta)d\beta,\quad \forall\varphi\in L^2[0,1],
\end{equation}
where $M(\cdot,\cdot)$ is a graphon. To ease notation, we are inclined to denote graphon operators by $M$ instead of $T_M$ when no confusion occurs. The graphon operator $M$ is self-adjoint and compact \cite{Lovasz}.

Then we may rewrite problem \eqref{4-1} as:
\begin{equation}\label{4-2}
 \left\{ \begin{split}
    dx_t^1= & \left\{A_1x_t^1+B_1u_t^1+D_1M_1x_t^1+\varepsilon F_1\bar{M}x_t^2\right\}dt+\sigma_1dW_t^1,\\
    dx_t^2= & \left\{A_2x_t^2+B_2u_t^2+D_2M_2x_t^2+\varepsilon F_2\bar{M}x_t^1\right\}dt+\sigma_2dW_t^2,
  \end{split}\right.
\end{equation}
where $M_1, M_2$ are graphon operators corresponding to $M_1(\cdot,\cdot), M_2(\cdot,\cdot)$, and $\bar{M}$ is the graphon operator corresponding to graphon $\bar{M}(\cdot,\cdot)=1$. It is easy to verify that $W=(W^1,W^2)$ is a $\mathcal{H}\oplus \mathcal{H}$-valued (or equivalently, $\mathcal{H}^2$-valued) $Q$-Wiener process. Indeed, from the series expansion of $W^1,W^2$, one has
\begin{equation*}
  W_t=\begin{pmatrix}
W_t^1 \\
W_t^2
\end{pmatrix}=\begin{pmatrix}
W_t^1 \\
0
\end{pmatrix}+\begin{pmatrix}
0 \\
W_t^2
\end{pmatrix}=
\sum_{k=1}^\infty \sqrt{\lambda_k^1}B_t^{1,k}\begin{pmatrix}
e_k^1 \\
0
\end{pmatrix}+\sum_{k=1}^\infty \sqrt{\lambda_k^2}B_t^{2,k}\begin{pmatrix}
0 \\
e_k^2
\end{pmatrix}.
\end{equation*}
Then, we have that $W_t$ is a $Q$-Wiener process with covariance operator $\mathbf{Q}=\mathbf{Q}_1 \oplus\mathbf{Q}_2$, $\left\{\begin{pmatrix}
e_k^1 \\
0
\end{pmatrix},\begin{pmatrix}
0 \\
e_k^2
\end{pmatrix}\right\}$ are orthonormal sets in $\mathcal{H}^2$ diagonalizing $\mathbf{Q}$, $\{\lambda_k^1,\lambda_k^2\}$ are the corresponding eigenvalues. Let
\begin{equation*}
  x_t=\begin{pmatrix}
x_t^1 \\
x_t^2
\end{pmatrix}, \ \mathbf{A}_\varepsilon=\begin{pmatrix} A_1I+D_1M_1 & \varepsilon F_1\bar{M} \\ \varepsilon F_2\bar{M} & A_2I+D_2M_2 \end{pmatrix},\
\mathbf{B}_1=\begin{pmatrix}
B_1I \\
0
\end{pmatrix},\
\mathbf{B}_2=\begin{pmatrix}
0 \\
B_2I
\end{pmatrix},\ \mathbf{\Sigma}=\begin{pmatrix}
\sigma_1I &  \\
  & \sigma_2I
\end{pmatrix},
\end{equation*}
where $I$ is the identity operator on $\mathcal{H}$. Therefore, one can treat \eqref{4-2} as the following stochastic evolution
\begin{equation}\label{4-3}
  dx_t=(\mathbf{A}_\varepsilon x_t+\mathbf{B}_1u_t^1+\mathbf{B}_2u_t^2)dt+\mathbf{\Sigma}dW_t.
\end{equation}
Similar to the proof of [47, Th.3.14], we can show that the equation \eqref{4-3} admits a unique strong (mild) solution.

Obviously,
\begin{align*}
   J^1(u^1,u^2)= & \frac{1}{2}E\int_0^T\big\{\bar{Q}_1\|(I-\Gamma_1M_1)x_t^1\|^2+R_{11}\|u_t^1\|^2+\varepsilon R_{12}\|u_t^2\|^2\big\}dt+\frac{1}{2}E\bar{Q}_{1f}\|x_T^1\|^2\\
  = & E\int_0^T\left\{\Big\langle \frac{1}{2}\bar{Q}_1\big((I-\Gamma_1M_1)^2\oplus 0\big)x_t,x_t\Big\rangle_{\mathcal{H}^2}+\Big\langle \frac{1}{2}R_{11}u_t^1,u_t^1\Big\rangle\right.\\
  & \left.+\Big\langle \frac{1}{2}\varepsilon R_{12}u_t^2,u_t^2\Big\rangle\right\}dt+E\Big\langle \frac{1}{2}\bar{Q}_{1f}(I\oplus 0)x_T,x_T\Big\rangle_{\mathcal{H}^2},
\end{align*}
where $0$ is the zero operator. Similarly, the cost function of the second team is given by
\begin{align*}
  J^2(u^1,u^2)= & E\int_0^T\left\{\Big\langle \frac{1}{2}\bar{Q}_2\big(0\oplus(I-\Gamma_2M_2)^2\big)x_t,x_t\Big\rangle_{\mathcal{H}^2}+\Big\langle \frac{1}{2}R_{22}u_t^2,u_t^2\Big\rangle\right.\\
  & \left.+\Big\langle \frac{1}{2}\varepsilon R_{21}u_t^1,u_t^1\Big\rangle\right\}dt+E\Big\langle \frac{1}{2}\bar{Q}_{2f}(0\oplus I)x_T,x_T\Big\rangle_{\mathcal{H}^2}.
\end{align*}

Therefore, the two-team Nash game problem is equivalent to an infinite dimensional two-agent Nash game problem. When $\varepsilon=1$, Problem \ref{twonash} is a standard infinite dimensional two-agent Nash game problem. When $\varepsilon$ lies in a small interval around zero, Problem \ref{twonash} is an infinite dimensional two-agent Nash game with weakly coupled agents. When $\varepsilon=0$, Problem \ref{twonash} degenerates into two independent single decision maker problems.

We now employ the dynamic programming approach to derive the two-team Nash equilibrium solution for Problem \ref{twonash}. Define the operator $\mathcal{L}_{u^1u^2}$ as follows:
\begin{equation*}
  \mathcal{L}_{u^1u^2}V(x)=\langle V_x(x),\mathbf{A}_\varepsilon x+\mathbf{B}_1u^1+\mathbf{B}_2u^2\rangle_{\mathcal{H}^2}+\frac{1}{2}tr[V_{xx}(x)\mathbf{\Sigma} \mathbf{Q}\mathbf{\Sigma}^*],
\end{equation*}
for any $u^1,u^2\in \mathcal{H}$, $x\in\mathcal{H}^2$ and twice Fr\'{e}chet differentiable real function $V(x)$ on $\mathcal{H}^2$. Then we propose sufficient conditions for Nash equilibrium in the following proposition.
\begin{proposition}
Suppose that there exist feedback strategies $\bar{u}^i=\bar{K}_i(t,x)$, and real function $V^i(t,x):[0,T]\times\mathcal{H}^2\rightarrow \mathbb{R}$, $i=1,2$ such that
\begin{itemize}
    \item[(i)] For each $t\in [0,T]$, $V^i(t,x)$ is twice Fr\'{e}chet differentiable in $x$, and for each $x\in\mathcal{H}^2$, $V^i(t,x)$ is differentiable in $t$;
    \item[(ii)] For any $x\in\mathcal{H}^2$,
     $$
     V^1(T,x)=\langle \frac{1}{2}\bar{Q}_{1f}(I\oplus 0)x,x\rangle_{\mathcal{H}^2},\ V^2(T,x)=\langle \frac{1}{2}\bar{Q}_{2f}(0\oplus I)x,x\rangle_{\mathcal{H}^2};
     $$
    \item[(iii)] For any $x\in\mathcal{H}^2$ and $u^1\in\mathcal{H}$,
    $$\begin{aligned}[t]
                 0=& V_t^1+\mathcal{L}_{\bar{u}^1\bar{u}^2}V^1+\Big\langle \frac{1}{2}\bar{Q}_1\big((I-\Gamma_1M_1)^2\oplus 0\big)x,x\Big\rangle_{\mathcal{H}^2}+\Big\langle \frac{1}{2}R_{11}\bar{u}^1,\bar{u}^1\Big\rangle_\mathcal{H}+\Big\langle \frac{1}{2}\varepsilon R_{12}\bar{u}^2,\bar{u}^2\Big\rangle_\mathcal{H}\\
  \leq & V_t^1+\mathcal{L}_{u^1\bar{u}^2}V^1+\Big\langle \frac{1}{2}\bar{Q}_1\big((I-\Gamma_1M_1)^2\oplus 0\big)x,x\Big\rangle_{\mathcal{H}^2}+\Big\langle \frac{1}{2}R_{11}u^1,u^1\Big\rangle_\mathcal{H}+\Big\langle \frac{1}{2}\varepsilon R_{12}\bar{u}^2,\bar{u}^2\Big\rangle_\mathcal{H};
               \end{aligned}
  $$
  \item[(iv)] For any $x\in\mathcal{H}^2$ and $u^2\in\mathcal{H}$,
  $$
  \begin{aligned}[t]
                 0=& V_t^2+\mathcal{L}_{\bar{u}^1\bar{u}^2}V^2+\Big\langle \frac{1}{2}\bar{Q}_2\big(0\oplus(I-\Gamma_2M_2)^2\big)x,x\Big\rangle_{\mathcal{H}^2}+\Big\langle \frac{1}{2}R_{22}\bar{u}^2,\bar{u}^2\Big\rangle_\mathcal{H}+\Big\langle \frac{1}{2}\varepsilon R_{21}\bar{u}^1,\bar{u}^1\Big\rangle_\mathcal{H}\\
  \leq & V_t^2+\mathcal{L}_{\bar{u}^1u^2}V^2+\Big\langle \frac{1}{2}\bar{Q}_2\big(0\oplus(I-\Gamma_2M_2)^2\big)x,x\Big\rangle_{\mathcal{H}^2}+\Big\langle \frac{1}{2}R_{22}u^2,u^2\Big\rangle_\mathcal{H}+\Big\langle \frac{1}{2}\varepsilon R_{21}\bar{u}^1,\bar{u}^1\Big\rangle_\mathcal{H}.
               \end{aligned}
  $$
\end{itemize}
Then for $i=1,2$, the set of strategies $\bar{u}^i=\bar{K}_i(t,x)$ is a two-team Nash equilibrium for Problem \ref{twonash}, and the corresponding cost functions are $J^i=V^i(0,x_0)$ for $i=1,2$.
\end{proposition}
The proof of this proposition is similar to the one of Lemma 4.1 in \cite{Ichikawa}, and we present a full proof here for completeness.

\begin{proof}
Let $\bar{x}_t$ denote the solution of \eqref{4-3} corresponding to strategies $\bar{u}_t^1,\bar{u}_t^2$. By It\^{o}'s formula, we have
\begin{align}\label{unknown9}
  &\Big\langle \frac{1}{2}\bar{Q}_{1f}(I\oplus 0)\bar{x}_T,\bar{x}_T\Big\rangle_{\mathcal{H}^2}-V^1(0,x_0)\nonumber\\
  = & V^1(T,\bar{x}_T)-V^1(0,x_0)\nonumber\\
  = & \int_0^T\langle V_x^1(s,\bar{x}_s),\mathbf{\Sigma} dW_s\rangle_{\mathcal{H}^2}+\int_0^T\Big\{V_t^1(s,\bar{x}_s)+\langle V_x^1(s,\bar{x}_s),\mathbf{A}_\varepsilon\bar{x}_s+\mathbf{B}_1\bar{u}_s^1+\mathbf{B}_2\bar{u}_s^2\rangle_{\mathcal{H}^2}\nonumber\\
  & +\frac{1}{2}tr[V_{xx}^1(s,\bar{x}_s)\mathbf{\Sigma} \mathbf{Q}\mathbf{\Sigma}^*]\Big\}ds.
\end{align}
Taking expectation on both sides of \eqref{unknown9} and applying (iii), one has
\begin{align*}
  & E\Big\langle \frac{1}{2}\bar{Q}_{1f}(I\oplus 0)\bar{x}_T,\bar{x}_T\Big\rangle_{\mathcal{H}^2}-V^1(0,x_0)\\
  =& E\int_0^T\Big\{V_t^1(s,\bar{x}_s)+\langle V_x^1(s,\bar{x}_s),\mathbf{A}_\varepsilon\bar{x}_s+\mathbf{B}_1\bar{u}_s^1+\mathbf{B}_2\bar{u}_s^2\rangle_{\mathcal{H}^2}+\frac{1}{2}tr[V_{xx}^1(s,\bar{x}_s)\mathbf{\Sigma} \mathbf{Q}\mathbf{\Sigma}^*]\Big\}ds\\
  =& -E\int_0^T\Big\{\Big\langle \frac{1}{2}\bar{Q}_1\big((I-\Gamma_1M_1)^2\oplus 0\big)\bar{x}_s,\bar{x}_s\Big\rangle_{\mathcal{H}^2}+\Big\langle \frac{1}{2}R_{11}\bar{u}_s^1,\bar{u}_s^1\Big\rangle+\Big\langle \frac{1}{2}\varepsilon R_{12}\bar{u}_s^2,\bar{u}_s^2\Big\rangle\Big\}ds.
\end{align*}
Thus
\begin{align}\label{4-4}
  V^1(0,x_0)=& E\int_0^T\Big\{\Big\langle \frac{1}{2}\bar{Q}_1\big((I-\Gamma_1M_1)^2\oplus 0\big)\bar{x}_s,\bar{x}_s\Big\rangle_{\mathcal{H}^2}\nonumber\\
  & +\Big\langle \frac{1}{2}R_{11}\bar{u}_s^1,\bar{u}_s^1\Big\rangle+\Big\langle \frac{1}{2}\varepsilon R_{12}\bar{u}_s^2,\bar{u}_s^2\Big\rangle\Big\}ds+E\Big\langle \frac{1}{2}\bar{Q}_{1f}(I\oplus 0)\bar{x}_T,\bar{x}_T\Big\rangle_{\mathcal{H}^2}.
\end{align}
For any admissible strategy $u^1$, from It\^{o}'s formula and (iii), we obtain
\begin{align*}
  & E\langle \frac{1}{2}\bar{Q}_{1f}(I\oplus 0)x_T,x_T\rangle_{\mathcal{H}^2}-V^1(0,x_0)\\
  =& E\int_0^T\Big\{V_t^1(s,x_s)+\langle V_x^1(s,x_s),\mathbf{A}_\varepsilon x_s+\mathbf{B}_1u_s^1+\mathbf{B}_2\bar{u}_s^2\rangle_{\mathcal{H}^2}+\frac{1}{2}tr[V_{xx}^1(s,x_s)\mathbf{\Sigma} \mathbf{Q}\mathbf{\Sigma}^*]\Big\}ds\\
  \geq & -E\int_0^T\Big\{\Big\langle \frac{1}{2}\bar{Q}_1\big((I-\Gamma_1M_1)^2\oplus 0\big)x_s,x_s\Big\rangle_{\mathcal{H}^2}+\Big\langle \frac{1}{2}R_{11}u_s^1,u_s^1\Big\rangle+\Big\langle \frac{1}{2}\varepsilon R_{12}\bar{u}_s^2,\bar{u}_s^2\Big\rangle\Big\}ds.
\end{align*}
Then
\begin{align}\label{4-5}
  V^1(0,x_0)\leq & E\int_0^T\Big\{\Big\langle \frac{1}{2}\bar{Q}_1\big((I-\Gamma_1M_1)^2\oplus 0\big)x_s,x_s\Big\rangle_{\mathcal{H}^2}\nonumber\\
  & +\Big\langle \frac{1}{2}R_{11}u_s^1,u_s^1\Big\rangle+\Big\langle \frac{1}{2}\varepsilon R_{12}\bar{u}_s^2,\bar{u}_s^2\Big\rangle\Big\}ds+E\Big\langle \frac{1}{2}\bar{Q}_{1f}(I\oplus 0)x_T,x_T\Big\rangle_{\mathcal{H}^2}.
\end{align}
From \eqref{4-4}-\eqref{4-5}, one has $J^1(\bar{u}^1,\bar{u}^2)\leq J^1(u^1,\bar{u}^2)$.
Similarly, it holds that $J^2(\bar{u}^1,\bar{u}^2)\leq J^2(\bar{u}^1,u^2)$. This completes the proof.
\end{proof}

We seek functions $V^i(t,x)$, $i=1,2$ of the form
\begin{equation*}
  V^i(t,x)=\langle \Pi^i(t)x,x\rangle_{\mathcal{H}^2}+q^i(t),\ \Pi^i(t)\in\mathbb{S}(\mathcal{H}^2).
\end{equation*}
Then
\begin{equation*}
  \mathcal{L}_{u^1u^2}V^i(t,x)=\langle 2\Pi^i(t)x,\mathbf{A}_\varepsilon x+\mathbf{B}_1u^1+\mathbf{B}_2u^2\rangle_{\mathcal{H}^2}+tr[\Pi^i(t)\mathbf{\Sigma}\mathbf{Q}\mathbf{\Sigma}^*],\ x\in\mathcal{H}^2,u^1,u^2\in\mathcal{H}.
\end{equation*}
It follows from (iii) and (iv) that
\begin{equation}\label{4-6}
  \left\{
  \begin{split}
    &\bar{u}^1=-\frac{2}{R_{11}}\mathbf{B}_1^*\Pi^1(t)x, \\
      &0=\frac{d}{dt}\langle \Pi^1(t)x,x\rangle_{\mathcal{H}^2}+\Big\langle \frac{1}{2}\bar{Q}_1\big((I-\Gamma_1M_1)^2\oplus 0\big)x,x\Big\rangle_{\mathcal{H}^2}\\
      &\quad +\langle 2\Pi^1(t)x,\mathbf{A}_\varepsilon x+\mathbf{B}_1\bar{u}^1+\mathbf{B}_2\bar{u}^2\rangle_{\mathcal{H}^2}+\langle \frac{1}{2}R_{11}\bar{u}^1,\bar{u}^1\rangle+\langle \frac{1}{2}\varepsilon R_{12}\bar{u}^2,\bar{u}^2\rangle, \ x\in\mathcal{H}^2,\\
      &0=\frac{d}{dt}q^1(t)+tr[\Pi^1(t)\mathbf{\Sigma}\mathbf{Q}\mathbf{\Sigma}^*],\\
      &\Pi^1(T)=\frac{1}{2}\bar{Q}_{1f}(I\oplus 0),\quad q^1(T)=0
  \end{split}
  \right.
\end{equation}
and
\begin{equation}\label{4-7}
  \left\{
  \begin{split}
    &\bar{u}^2=-\frac{2}{R_{22}}\mathbf{B}_2^*\Pi^2(t)x, \\
      &0=\frac{d}{dt}\langle \Pi^2(t)x,x\rangle_{\mathcal{H}^2}+\Big\langle \frac{1}{2}\bar{Q}_2\big(0\oplus(I-\Gamma_2M_2)^2\big)x,x\Big\rangle_{\mathcal{H}^2}\\
      &\quad +\langle 2\Pi^2(t)x,\mathbf{A}_\varepsilon x+\mathbf{B}_1\bar{u}^1+\mathbf{B}_2\bar{u}^2\rangle_{\mathcal{H}^2}+\langle \frac{1}{2}R_{22}\bar{u}^2,\bar{u}^2\rangle+\langle \frac{1}{2}\varepsilon R_{21}\bar{u}^1,\bar{u}^1\rangle,\ x\in\mathcal{H}^2,\\
      &0=\frac{d}{dt}q^2(t)+tr[\Pi^2(t)\mathbf{\Sigma}\mathbf{Q}\mathbf{\Sigma}^*],\\
      &\Pi^2(T)=\frac{1}{2}\bar{Q}_{2f}(0\oplus I),\quad q^2(T)=0.
  \end{split}
  \right.
\end{equation}
Thus, from equations \eqref{4-6} and \eqref{4-7}, we can obtain the following result.
\begin{theorem}
Problem \ref{twonash} has the following two-team Nash equilibrium solution
\begin{equation*}
  \bar{u}^i=-\frac{2}{R_{ii}}\mathbf{B}_i^*\Pi^i(t)x,\ i=1,2,
\end{equation*}
and the corresponding cost functions are
\begin{equation*}
  J^i(\bar{u}^1,\bar{u}^2)=v^i(0,x_0)=\langle \Pi^i(0)x_0,x_0\rangle_{\mathcal{H}^2}+q^i(0),
\end{equation*}
when $\Pi^1(t)$ and $\Pi^2(t)$ satisfy the following set of coupled operator-valued Riccati-type equations
\begin{equation}\label{4-8}
\left\{
  \begin{split}
    0=&\frac{d}{dt}\langle \Pi^1(t)x,x\rangle_{\mathcal{H}^2}+\langle 2\Pi^1(t)\mathbf{A}_\varepsilon x,x\rangle_{\mathcal{H}^2}-\langle \frac{2}{R_{11}}\Pi^1(t)\mathbf{B}_1\mathbf{B}_1^*\Pi^1(t)x,x\rangle_{\mathcal{H}^2}\\
    &-\langle \frac{2}{R_{22}}\Pi^1(t)\mathbf{B}_2\mathbf{B}_2^*\Pi^2(t)x,x\rangle_{\mathcal{H}^2}-\langle \frac{2}{R_{22}}\Pi^2(t)\mathbf{B}_2\mathbf{B}_2^*\Pi^1(t)x,x\rangle_{\mathcal{H}^2}\\
    &+\langle \varepsilon\frac{2R_{12}}{R_{22}^2}\Pi^2(t)\mathbf{B}_2\mathbf{B}_2^*\Pi^2(t)x,x\rangle_{\mathcal{H}^2}\\
    &+\big\langle \frac{1}{2}\bar{Q}_1\big((I-\Gamma_1M_1)^2\oplus 0\big)x,x\big\rangle_{\mathcal{H}^2},\quad x\in\mathcal{H}^2,\\
    0=&\frac{d}{dt}\langle \Pi^2(t)x,x\rangle_{\mathcal{H}^2}+\langle 2\Pi^2(t)\mathbf{A}_\varepsilon x,x\rangle_{\mathcal{H}^2}-\langle \frac{2}{R_{22}}\Pi^2(t)\mathbf{B}_2\mathbf{B}_2^*\Pi^2(t)x,x\rangle_{\mathcal{H}^2}\\
    &-\langle \frac{2}{R_{11}}\Pi^2(t)\mathbf{B}_1\mathbf{B}_1^*\Pi^1(t)x,x\rangle_{\mathcal{H}^2}-\langle \frac{2}{R_{11}}\Pi^1(t)\mathbf{B}_1\mathbf{B}_1^*\Pi^2(t)x,x\rangle_{\mathcal{H}^2}\\
    &+\langle \varepsilon\frac{2R_{21}}{R_{11}^2}\Pi^1(t)\mathbf{B}_1\mathbf{B}_1^*\Pi^1(t)x,x\rangle_{\mathcal{H}^2}\\
    &+\langle \frac{1}{2}\bar{Q}_2\big(0\oplus(I-\Gamma_2M_2)^2\big)x,x\big\rangle_{\mathcal{H}^2},\quad x\in\mathcal{H}^2,\\
    \Pi^1(T) & =\frac{1}{2}\bar{Q}_{1f}(I\oplus 0),\quad\Pi^2(T)=\frac{1}{2}\bar{Q}_{2f}(0\oplus I).
  \end{split}
  \right.
\end{equation}
\end{theorem}

\section{Existence of solutions for operator-valued Riccati-type equations}
\qquad To our best knowledge, until now there have been no results ensuring the existence of the solutions of \eqref{4-8} on $[0,T]$. In this section, we will focus on the well-posedness of \eqref{4-8}. Let us begin with the following result.
\begin{lemma}
The space $\mathcal{D}=\{T_1\oplus T_2|T_1,T_2\in\mathbb{S}(\mathcal{H}^2)\}$, endowed with the norm in $\mathcal{L}(\mathcal{H}^4)$, is a Banach space.
\end{lemma}
\begin{proof}
Obviously, the set $\mathcal{D}$ is closed under linear operations. Let $\{T_1^n\oplus T_2^n\}\subset \mathcal{D}$ be a convergent sequence. Since the space $\mathbb{S}(\mathcal{H}^2\oplus \mathcal{H}^2)$ is complete, there exists an element $T=\begin{pmatrix} T_{11} & T_{12} \\ (T_{12})^* & T_{22} \end{pmatrix}\in\mathbb{S}(\mathcal{H}^2\oplus\mathcal{H}^2)$ with $T_{11},T_{22}\in\mathbb{S}(\mathcal{H}^2)$ and $T_{12}\in\mathcal{L}(\mathcal{H}^2)$, such that
\begin{equation}\label{unknown6}
  \lim_{n\rightarrow\infty}\|T_1^n\oplus T_2^n-T\|_{\mathcal{L}(\mathcal{H}^4)}=0.
\end{equation}
The goal is to show that $T\in\mathcal{D}$, i.e., $T_{12}=0$. From \eqref{unknown6}, for any $\epsilon>0$, $\exists N_0$, such that for $n>N_0$, one has $\|T_1^n\oplus T_2^n-T\|_{\mathcal{L}(\mathcal{H}^4)}<\epsilon$. Then, for $n>N_0$, and any $\|u\|_{\mathcal{H}^2}=1$,
\begin{align*}
  \|(T_1^n-T_{11})u\|_{\mathcal{H}^2}^2+\|T_{12}^*u\|_{\mathcal{H}^2}^2=\|(T_1^n\oplus T_2^n-T)(u,0)\|_{\mathcal{H}^4}^2<\epsilon^2,
\end{align*}
which means $T_{12}=0$. Thus, we have $T\in\mathcal{D}$, which yields that $\mathcal{D}$ is a closed subspace of $\mathbb{S}(\mathcal{H}^2\oplus\mathcal{H}^2)$. Since $\mathbb{S}(\mathcal{H}^2\oplus\mathcal{H}^2)$ is a Banach space, one has that $\mathcal{D}$ is a Banach space with the norm inherited from $\mathcal{L}(\mathcal{H}^4)$. In addition, it is easy to see that, in subspace $\mathcal{D}$, the norm of $\mathcal{L}(\mathcal{H}^4)$ is equal to $\|T_1\oplus T_2\|_\mathcal{D}=\max\{\|T_1\|_{\mathcal{L}(\mathcal{H}^2)},\|T_2\|_{\mathcal{L}(\mathcal{H}^2)}\}$.
\end{proof}
Set
\begin{equation*}
\mathbf{\Pi}(t)=\begin{pmatrix} \Pi^1(t) & \\ & \Pi^2(t) \end{pmatrix},\ \mathbf{K}_\varepsilon=\begin{pmatrix} \mathbf{A}_\varepsilon & \\ & \mathbf{A}_\varepsilon \end{pmatrix},\ \mathbf{S}=\begin{pmatrix} \frac{2}{R_{11}}\mathbf{B}_1\mathbf{B}_1^* & \\ & \frac{2}{R_{22}}\mathbf{B}_2\mathbf{B}_2^* \end{pmatrix},
\end{equation*}
\begin{equation*}
  \mathbf{S}_0=\begin{pmatrix} \frac{2R_{12}}{(R_{22})^2}\mathbf{B}_2\mathbf{B}_2^* & \\ & \frac{2R_{21}}{(R_{11})^2}\mathbf{B}_1\mathbf{B}_1^* \end{pmatrix},\ \mathbf{I}_{2\times 2}=\begin{pmatrix} I & \\ & I \end{pmatrix},\ \mathbf{J}=\begin{pmatrix}  & \mathbf{I}_{2\times 2} \\ \mathbf{I}_{2\times 2} & \end{pmatrix},
\end{equation*}
\begin{equation*}
  \bar{\mathbf{Q}}=\begin{pmatrix} \frac{1}{2}\bar{Q}_1\big((I-\Gamma_1M_1)^2\oplus 0\big) & \\ & \frac{1}{2}\bar{Q}_2\big(0\oplus(I-\Gamma_2M_2)^2\big) \end{pmatrix},
\ \mathbf{G}=\begin{pmatrix} \frac{1}{2}\bar{Q}_{1f}(I\oplus 0) & \\ & \frac{1}{2}\bar{Q}_{2f}(0\oplus I) \end{pmatrix}.
\end{equation*}
Then it holds that
\begin{equation*}
  \mathbf{K}_\varepsilon,\mathbf{S},\mathbf{S}_0,\mathbf{J},\bar{\mathbf{Q}},\mathbf{G}\in\mathcal{L}(\mathcal{H}^4).
\end{equation*}
Clearly,  we can check that $\mathbf{\Pi}(\cdot)$ satisfies \eqref{4-8} if $\mathbf{\Pi}(t)\in\mathcal{D}$ for $t\in [0,T]$ and $\mathbf{\Pi}(\cdot)$ satisfies the following equation
\begin{equation}\label{4-9}
\left\{
  \begin{split}
    & 0=\frac{d}{dt}\mathbf{\Pi}+\mathbf{\Pi}\mathbf{K}_\varepsilon+\mathbf{K}_\varepsilon^*\mathbf{\Pi}-\mathbf{\Pi}\mathbf{S}\mathbf{\Pi}-\mathbf{\Pi}\mathbf{J}\mathbf{S}\mathbf{\Pi}\mathbf{J}-\mathbf{J}\mathbf{\Pi}\mathbf{S}\mathbf{J}\mathbf{\Pi}+\varepsilon\mathbf{J}\mathbf{\Pi}\mathbf{J}\mathbf{S}_0\mathbf{J}\mathbf{\Pi}\mathbf{J}+\bar{\mathbf{Q}},\\
    &\mathbf{\Pi}(T)=\mathbf{G}.
  \end{split}
  \right.
\end{equation}

Noting that $\mathbf{A}_\varepsilon$ is bounded, we only need to study the uniform solution of \eqref{4-9} here.
When $\varepsilon=1$, we have the following result.
\begin{theorem}\label{th3.1}
Let $r=2\|G\|_{\mathcal{L}(\mathcal{H}^4)}$, $\alpha\in (0,1)$, and $\tau\in (0,+\infty)$ such that
$$
\tau\big(2r\|K_1\|_{\mathcal{L}(\mathcal{H}^4)}+r^2\|\mathbf{S}\|_{\mathcal{L}(\mathcal{H}^4)}+2r^2\|\mathbf{J}\|_{\mathcal{L}(\mathcal{H}^4)}^2\|\mathbf{S}\|_{\mathcal{L}(\mathcal{H}^4)}+r^2\|\mathbf{J}\|_{\mathcal{L}(\mathcal{H}^4)}^4\|\mathbf{S}_0\|_{\mathcal{L}(\mathcal{H}^4)}\\
+\|\bar{\mathbf{Q}}\|_{\mathcal{L}(\mathcal{H}^4)}\big)\leq \|G\|_{\mathcal{L}(\mathcal{H}^4)}
$$
and
$$
\tau\big(2\|K_1\|_{\mathcal{L}(\mathcal{H}^4)}+2r\|\mathbf{S}\|_{\mathcal{L}(\mathcal{H}^4)}
+4r\|\mathbf{J}\|_{\mathcal{L}(\mathcal{H}^4)}^2\|\mathbf{S}\|_{\mathcal{L}(\mathcal{H}^4)}
+2r\|\mathbf{J}\|_{\mathcal{L}(\mathcal{H}^4)}^4\|\mathbf{S}_0\|_{\mathcal{L}(\mathcal{H}^4)}\big)\leq\alpha.
$$
Then \eqref{4-9} has unique solution in $B_{r,\tau}=\{F\in C^1([T-\tau,T];\mathcal{D}): \|F\|_{\mathcal{L}(\mathcal{H}^4)}\leq r,\forall t\in [T-\tau,T]\}$.
\end{theorem}
\begin{proof}
The proof is similar to the one of \cite[Th 2.1]{guo} and so we omit it here.
\end{proof}
\begin{remark}
Theorem \ref{th3.1} shows that equation \eqref{4-9} admits a unique solution in a sufficiently small time interval when $\varepsilon=1$.
\end{remark}

For sufficiently small $\varepsilon$, we can obtain the following result.
\begin{theorem}\label{existencethm}
There exists $\varepsilon_0>0$, such that for $\varepsilon\in [-\varepsilon_0,\varepsilon_0]$, equation \eqref{4-9} admits a solution in $C^1([0,T];\mathcal{D})$.
\end{theorem}
Before proving this theorem, we first present another result, which is a parallel extension of \cite[Th 7.2]{Tikhonov} to the Banach space setting.
\begin{theorem}\label{tuiguang}
Let $X$ be a Banach space, $G=\{(y,t,\varepsilon)|\|y\|_X<b,t\in \mathbb{R},|\varepsilon|\leq\bar{\varepsilon}\}$, function $f(y,t,\varepsilon):G\rightarrow X$ be continuous with respect to its variables, and satisfy the Lipschitz condition
\begin{equation}\label{4-10}
  \|f(y_1,t,\varepsilon)-f(y_2,t,\varepsilon)\|_X\leq N\|y_1-y_2\|_X,
\end{equation}
where $N$ is the same constant for all $\varepsilon\in [-\bar{\varepsilon},\bar{\varepsilon}]$ and $t\in\mathbb{R}$. Assume that the following initial value problem
\begin{equation*}
  \frac{dy}{dt}=f(y,t,0),\quad y(T)=y_T\in X,
\end{equation*}
admits a solution $\bar{y}(t)\in D=\{(y,t)| \|y\|_X<b, 0\leq t\leq T\}$ on the interval $[0,T]$. Then there exists sufficiently small $\varepsilon_0>0$ such that, for $\varepsilon\in [-\varepsilon_0,\varepsilon_0]$, the following equation
\begin{equation*}
  \frac{dy}{dt}=f(y,t,\varepsilon),\quad y(T,\varepsilon)=y_T\in X,
\end{equation*}
also admits a solution $y(t,\varepsilon)\in D$ on $[0,T]$.
\end{theorem}
\begin{proof}
For any given $\varepsilon\in [-\bar{\varepsilon},\bar{\varepsilon}]$, consider the following initial value problem
\begin{equation}\label{4-11}
  \frac{d\Lambda}{dt}=f(\bar{y}+\Lambda,t,\varepsilon)-f(\bar{y},t,\varepsilon)+f(\bar{y},t,\varepsilon)-f(\bar{y},t,0),\quad \Lambda(T,\varepsilon)=0,
\end{equation}
and the following domain of variables $(\Lambda,t)$: $\widetilde{D}=\{(\Lambda,t)|\|\Lambda\|_X<C, 0\leq t\leq T\}$, with $C=b-\beta$, $\beta=\sup_{t\in [0,T]}\|\bar{y}(t)\|_X$. It is easy to check that $\beta<b$. For $\|\Lambda(t,\varepsilon)\|_X<C$, one has $\|\bar{y}(t)+\Lambda(t,\varepsilon)\|_X<b$, i.e., the variables $(\bar{y}+\Lambda,t)$ of $f(\bar{y}+\Lambda,t,\varepsilon)$ belong to the domain $D$.

Since $f(y,t,\varepsilon)$ is continuous in $G$, one has that $F(t,\varepsilon)\triangleq \|f(\bar{y}(t),t,\varepsilon)-f(\bar{y}(t),t,0)\|_X$ is continuous in $[0,T]\times [-\bar{\varepsilon},\bar{\varepsilon}]$, which implies that $F(t,\varepsilon)$ is uniformly continuous on $[0,T]\times [-\bar{\varepsilon},\bar{\varepsilon}]$. Then, for any $\epsilon>0$, there exists $\delta>0$, such that for $|\varepsilon|<\delta$, and all $t\in [0,T]$, we have
\begin{equation*}
  |F(t,\varepsilon)-F(t,0)|=\|f(\bar{y}(t),t,\varepsilon)-f(\bar{y}(t),t,0)\|_X<\epsilon,
\end{equation*}
which leads to
$$
\omega(\varepsilon)\triangleq \sup_{t\in [0,T]}\|f(\bar{y}(t),t,\varepsilon)-f(\bar{y}(t),t,0)\|_X\leq\epsilon.
$$
Thus, one has $\omega(\varepsilon)\rightarrow 0$ as $\varepsilon\rightarrow 0$.

For any given $\varepsilon\in [-\bar{\varepsilon},\bar{\varepsilon}]$, set $h(\Lambda,t,\varepsilon)=f(\bar{y}(t)+\Lambda,t,\varepsilon)-f(\bar{y}(t),t,0)$. By \eqref{4-10}-\eqref{4-11}, it is easy to see that $\Lambda\mapsto h(\Lambda,t,\varepsilon)$ is Lipschitz continuous and $h(\cdot,\cdot,\varepsilon)$ is continuous on $\widetilde{D}$. Then, by \cite[Th 2.1]{guo}, for any given constant $0<C_1<C$, the problem \eqref{4-11} admits a unique solution $\Lambda(t,\varepsilon)$ on a certain segment $[T-H,T]$, which satisfies $\|\Lambda\|_X\leq C_1$. Denote $\widetilde{\widetilde{D}}=\{(\Lambda,t)|\|\Lambda\|_X<C,t\in \mathbb{R}\}$,
\begin{equation*}
  \bar{\bar{y}}(t)=\left\{\begin{split}
                     & \bar{y}(0),\quad & t<0,\\
                     & \bar{y}(t),\quad & 0\leq t\leq T,\\
                     & y_T,\quad & t>T,
                   \end{split}\right.
\end{equation*}
and $\bar{h}(\Lambda,t,\varepsilon)=f(\bar{\bar{y}}(t)+\Lambda,t,\varepsilon)-f(\bar{\bar{y}}(t),t,0)$, respectively. Obviously, one has that $h=\bar{h}$ on $\widetilde{D}$. For the following problem
\begin{equation}\label{5.3.5a}
  \frac{d\Lambda}{dt}=\bar{h}(\Lambda,t,\varepsilon),\quad \Lambda(T,\varepsilon)=0,
\end{equation}
it is easy to see that $\bar{h}(\cdot,\cdot,\varepsilon)$ is continuous and $\Lambda\mapsto \bar{h}(\Lambda,t,\varepsilon)$ is Lipschitz continuous on $\widetilde{\widetilde{D}}$. By \cite[Th 2.2]{guo}, the solution of problem \eqref{5.3.5a} has a unique left continuation to a maximal interval $(T-\bar{H},T]$. Suppose that $\bar{H}\leq T$. Now we show that the limit $\lim\limits_{t\rightarrow(T-\bar{H})+0}\Lambda(t,\varepsilon)$ exists.

Let $\{H_n\}$ be any sequence satisfying $H_n\in (T-\bar{H},T-H]$ and $H_n\downarrow (T-\bar{H})$. Rewrite equation \eqref{5.3.5a} in integral form, we can obtain
\begin{align}\label{unknownxin}
  & \|\Lambda(H_{n+m},\varepsilon)-\Lambda(H_{n},\varepsilon)\|_X\nonumber\\
  \leq & \int_{H_{n+m}}^{H_{n}}\big[\|f(\bar{y}+\Lambda,s,\varepsilon)-f(\bar{y},s,\varepsilon)\|_X+\|f(\bar{y},s,\varepsilon)-f(\bar{y},s,0)\|_X\big]ds\nonumber\\
  \leq & \int_{H_{n+m}}^{H_{n}}\big[N\|\Lambda(s,\varepsilon)\|_X+\omega(\varepsilon)\big]ds\nonumber\\
  \leq & \big[NC+\omega(\varepsilon)\big](H_{n}-H_{n+m}).
\end{align}
Then, $\{\Lambda(H_{n},\varepsilon)\}$ is a Cauchy sequence. Since $X$ is a Banach space, $\Lambda(H_n,\varepsilon)$ converges to a certain limit $\bar{\Lambda}(\varepsilon)\in X$, i.e., $\lim\limits_{n\rightarrow\infty}\|\Lambda(H_n,\varepsilon)-\bar{\Lambda}(\varepsilon)\|_X=0$.
For any sequence $\{H_n\}$ satisfying the above mentioned conditions, the limits of $\{\Lambda(H_n,\varepsilon)\}$ are identical.
Indeed, assume that $\{H_n^i\}$, $i=1,2$ are sequences satisfy $H_n^i\downarrow (T-\bar{H})$. Denote $\bar{\Lambda}^i(\varepsilon)$ be the limit of $\{\Lambda(H_n^i,\varepsilon)\}$. Then, for any $\epsilon>0$, there exists $N_0$ such that, for $n>N_0$, one has $$\|\bar{\Lambda}^i(\varepsilon)-\Lambda(H_n^i,\varepsilon)\|_X<\frac{\epsilon}{3}, \; i=1,2, \quad |H_n^1-H_n^2|<\frac{\epsilon}{3(NC+\omega(\varepsilon))}.$$
Similar to \eqref{unknownxin}, we have
\begin{align*}
  & \|\bar{\Lambda}^1(\varepsilon)-\bar{\Lambda}^2(\varepsilon)\|_X\\ = & \|\bar{\Lambda}^1(\varepsilon)-\Lambda(H_n^1,\varepsilon)+\Lambda(H_n^1,\varepsilon)-\Lambda(H_n^2,\varepsilon)+\Lambda(H_n^2,\varepsilon)-\bar{\Lambda}^2(\varepsilon)\|_X\\
  \leq & \|\bar{\Lambda}^1(\varepsilon)-\Lambda(H_n^1,\varepsilon)\|_X+\|\Lambda(H_n^1,\varepsilon)-\Lambda(H_n^2,\varepsilon)\|_X+\|\Lambda(H_n^2,\varepsilon)-\bar{\Lambda}^2(\varepsilon)\|_X\\
  < & \epsilon.
\end{align*}
As $\epsilon$ is arbitrary, $\bar{\Lambda}^1(\varepsilon)=\bar{\Lambda}^2(\varepsilon)$ holds.
Then, by Heine theorem, we obtain $\lim\limits_{t\rightarrow(T-\bar{H})+0}\Lambda(t,\varepsilon)=\bar{\Lambda}(\varepsilon)$. By \cite[Th 2.2]{guo}, one has $\|\bar{\Lambda}(\varepsilon)\|_X=C$.

Denote
\begin{equation*}
  \tilde{\Lambda}(t,\varepsilon)=\left\{\begin{split}
                                   & \bar{\Lambda}(\varepsilon).\ t=T-\bar{H}, \\
                                     & \Lambda(t,\varepsilon),\ T-\bar{H}<t\leq T.
                                 \end{split}\right.
\end{equation*}
Then $\tilde{\Lambda}(t,\varepsilon)$ is continuous on $[T-\bar{H},T]$. Equation \eqref{5.3.5a} and continuity of $\tilde{\Lambda}(t,\varepsilon)$ yield
\begin{equation*}
  -\tilde{\Lambda}(t,\varepsilon)=\int_t^T[f(\bar{y}+\tilde{\Lambda},s,\varepsilon)-f(\bar{y},s,0)]ds,\quad t\in [T-\bar{H},T].
\end{equation*}
By Gronwall's lemma,
\begin{align*}
  \|\tilde{\Lambda}(t,\varepsilon)\|_X\leq & \int_t^T\big[\|f(\bar{y}+\tilde{\Lambda},s,\varepsilon)-f(\bar{y},s,\varepsilon)\|_X+\|f(\bar{y},s,\varepsilon)-f(\bar{y},s,0)\|_X\big]ds\\
  \leq & \int_t^T\big[N\|\tilde{\Lambda}(s,\varepsilon)\|_X+\omega(\varepsilon)\big]ds\\
  \leq & T\omega(\varepsilon)e^{NT}.
\end{align*}
Thus, there exists $\varepsilon_0>0$, such that for $|\varepsilon|\leq\varepsilon_0$, one has $T\omega(\varepsilon)e^{NT}<C$, which implies
$$\|\bar{\Lambda}(\varepsilon)\|_X=\|\tilde{\Lambda}(T-\bar{H},\varepsilon)\|_X<C.$$
Then, for $|\varepsilon|\leq\varepsilon_0$, we obtain a contradiction. As a result, for $|\varepsilon|\leq\varepsilon_0$, one has $\bar{H}>T$. Thus, problem \eqref{5.3.5a} admits a unique solution on $\widetilde{D}$, which means that problem \eqref{4-11} has a unique solution on $\widetilde{D}$. This further yields that $\bar{y}(t)+\Lambda(t,\varepsilon)$ is a solution for the original problem.
\end{proof}

We now provide the proof of Theorem \ref{existencethm}.

\begin{proof}
It is easy to see that the RHS of \eqref{4-9} defines a function from $\mathcal{D}\times [-\bar{\varepsilon},\bar{\varepsilon}]$ to $\mathcal{D}$, which is continuous with respect to $(\Pi,\varepsilon)$ and satisfies \eqref{4-10}. From Theorem \ref{tuiguang}, we only need to show that \eqref{4-9} admits a solution on $[0,T]$ for the case $\varepsilon=0$.

Denote $\Pi^i=\begin{pmatrix} \Pi_{11}^i &\Pi_{12}^i\\ (\Pi_{12}^i)^* &\Pi_{22}^i\end{pmatrix}$. From \eqref{4-9}, for $\varepsilon=0$, one has
\begin{align*}
  0= & \frac{d}{dt}\begin{pmatrix} \Pi_{11}^1 &\Pi_{12}^1\\ (\Pi_{12}^1)^* &\Pi_{22}^1\end{pmatrix}+\begin{pmatrix} \Pi_{11}^1(A_1I+D_1M_1) & \Pi_{12}^1(A_2I+D_2M_2) \\ (\Pi_{12}^1)^*(A_1I+D_1M_1)& \Pi_{22}^1(A_2I+D_2M_2) \end{pmatrix}\\
  &+\begin{pmatrix} (A_1I+D_1M_1)\Pi_{11}^1 & (A_1I+D_1M_1)\Pi_{12}^1 \\ (A_2I+D_2M_2)(\Pi_{12}^1)^* & (A_2I+D_2M_2)\Pi_{22}^1 \end{pmatrix}\\
  &-\frac{2(B_1)^2}{R_{11}}\begin{pmatrix} (\Pi_{11}^1)^2 & \Pi_{11}^1\Pi_{12}^1\\ (\Pi_{12}^1)^*\Pi_{11}^1 & (\Pi_{12}^i)^*\Pi_{12}^1\end{pmatrix}-\frac{2(B_2)^2}{R_{22}}\begin{pmatrix} \Pi_{12}^1(\Pi_{12}^2)^* & \Pi_{12}^1\Pi_{22}^2\\ \Pi_{22}^1(\Pi_{12}^2)^* & \Pi_{22}^1\Pi_{22}^2\end{pmatrix}\\
  &-\frac{2(B_2)^2}{R_{22}}\begin{pmatrix} \Pi_{12}^2(\Pi_{12}^1)^* & \Pi_{12}^2\Pi_{22}^1\\ \Pi_{22}^2(\Pi_{12}^1)^* & \Pi_{22}^2\Pi_{22}^1\end{pmatrix}+\frac{1}{2}\bar{Q}_1\begin{pmatrix} (I-\Gamma_1M_1)^2 & 0 \\ 0 & 0 \end{pmatrix}
\end{align*}
and
$$
\begin{pmatrix} \Pi_{11}^1(T) &\Pi_{12}^1(T) \\ (\Pi_{12}^1(T))^* &\Pi_{22}^1(T) \end{pmatrix}=\frac{1}{2}\bar{Q}_{1f}\begin{pmatrix} I & 0 \\ 0 & 0 \end{pmatrix}.
$$
A similar result holds for $\Pi^2$. By observing the form of the equation, we can find a particular solution $\bar{\Pi}=\begin{pmatrix} \bar{\Pi}^1 &  \\  & \bar{\Pi}^2 \end{pmatrix}\in \mathcal{D}$ to \eqref{4-9}, where $\bar{\Pi}_{11}^1$, $\bar{\Pi}_{22}^2$ are the unique solutions of the following operator-valued Riccati equations for $i=1,2$, respectively(see \cite[Part IV, Chapter 1, Th 2.1]{Bensoussan2} for the existence and uniqueness of the equations):
\begin{equation*}
\left\{
\begin{split}
  &0=\frac{d}{dt}\Pi_{ii}^i+\Pi_{ii}^i(A_iI+D_iM_i)+(A_iI+D_iM_i)\Pi_{ii}^i-\frac{2(B_i)^2}{R_{ii}}(\Pi_{ii}^i)^2+\frac{1}{2}\bar{Q}_i(I-\Gamma_iM_i)^2,\\
  &\Pi_{ii}^i(T)=\frac{1}{2}\bar{Q}_{if}I,
\end{split}
\right.
\end{equation*}
and $\bar{\Pi}_{12}^1,\bar{\Pi}_{22}^1,\bar{\Pi}_{11}^2,\bar{\Pi}_{12}^2=0$. Then, equation \eqref{4-9} admits a solution on $[0,T]$ for the case $\varepsilon=0$. By Theorem \ref{tuiguang}, there exists $\varepsilon_0>0$, such that for $\varepsilon\in [-\varepsilon_0,\varepsilon_0]$, equation \eqref{4-9} has a solution on $[0,T]$.
\end{proof}

\section{Conclusions and future work}
\qquad This article is concerned with the study of the nonzero-sum stochastic differential games between two graphon teams. A Nash equilibrium for the two-team Nash game is derived through the dynamic programming approach, and the existence of solutions of the corresponding coupled operator-valued Riccati-type equations is obtained. It is worth noticing that, in our work, we utilize nonlocal diffusion equations to describe the dynamics of populations, where each contains a continuum of agents. Thus, it would be interesting to discuss the two-team Nash game problem based on the model appears in \cite{Coppini}. We leave this as our future work.

\section*{Acknowledgement}
\qquad The authors would like to thank Prof. Qi L\"{u} and Ms. Yue Zeng for their helpful suggestions and discussions.

\end{document}